\documentclass[11pt,a4paper,reqno]{article}
\usepackage[latin1]{inputenc}
\usepackage[T1]{fontenc}
\usepackage[english]{babel}
\usepackage{graphicx}
\usepackage{amsfonts,amssymb,amsbsy,latexsym,a4wide,xspace}
\usepackage{amsmath,delarray,enumerate,calc,amsthm}
\usepackage{bbm}
\usepackage{mathrsfs}
\usepackage{txfonts}
\usepackage{paralist}

\usepackage{authblk}

\usepackage{tikz}
\usetikzlibrary{matrix}

\usepackage{url}

\newcommand\blfootnote[1]{%
  \begingroup
  \renewcommand\thefootnote{}\footnote{#1}%
  \addtocounter{footnote}{-1}%
  \endgroup
}

\newcommand{\supp}{\operatorname{supp}}

\newcommand{\cA}{{\mathcal A}}
\newcommand{\cC}{{\mathcal C}}
\newcommand{\cB}{{\mathcal B}}
\newcommand{\cE}{{\mathcal E}}
\newcommand{\cF}{{\mathcal F}}

\newcommand{\cD}{{\mathcal D}}

\newcommand{\cP}{{\mathcal P}}

\newcommand{\cM}{{\mathcal M}}

\newcommand{\N}{{\mathbbm N}}

\newcommand{\Z}{{\mathbbm Z}}

\newcommand{\inn}{\operatorname{int}}
\newcommand{\lcm}{\operatorname{lcm}}

\newcommand{\card}{\operatorname{card}}

\newcommand{\piH}{\pi^{\scriptscriptstyle H}}

\newcommand{\ddelta}{\boldsymbol{\delta}}
\newcommand{\prim}[1]{#1^{prim}}
\newcommand{\Ainf}{\cA_\infty}
\newcommand{\Spec}{\operatorname{Spec}}
\renewcommand{\mod}{\text{ mod }}
\renewcommand{\sigma}{S}
\newtheorem {lemma}{Lemma}
\newtheorem{theorem}{Theorem}
\newtheorem {bemerkung}{Remark}
\newtheorem{proposition}{Proposition}
\newtheorem {corollary}{Corollary}
\newtheorem{beispiel}{Example}
\newtheorem{frage}{Question}
\newtheorem{vermutung}{Conjecture}
\newenvironment{question} {\begin{frage} \normalfont }{\end{frage}}
\newenvironment{conjecture} {\begin{vermutung} \normalfont }{\end{vermutung}}
\newenvironment{remark} {\begin{bemerkung} \normalfont }{\end{bemerkung}}

\newcommand{\strip}{[\varphi_{W'},\varphi]}

\begin{document}

\title{Generalized heredity in $\cB$-free systems
}
\author{Gerhard Keller
\thanks{I am indebted to Aurelia Dymek for several very helpful remarks and corrections to a previous version of this manuscript.}
\
\thanks{
The final version of these notes is based upon work supported by the Swedish Research Council under grant no. 2016-06596 while the author was in residence at Institut Mittag-Leffler in Djursholm, Sweden during the workshop
``Thermodynamic Formalism -- Applications to Geometry, Number Theory, and Stochastics'' (2019).}
}
\affil{Department of Mathematics, University of Erlangen-N\"urnberg, \\
 Cauerstr. 11, 91058 Erlangen, Germany
 \par Email: keller@math.fau.de}
\date{Version of \today}


\maketitle

{\begin{abstract}
Let $\cB\subseteq\N$ be a primitive set. We complement results on heredity of the $\cB$-free subshift $X_\eta$ from \cite{BKKL2015} in two directions: In the proximal case we prove that a subshift $X_\varphi$, which might be slightly larger than the subshift $X_\eta$, is always hereditary. (There is no need to assume that the set $\cB$ has light tails as in \cite{BKKL2015}, but if $\cB$ is taut, in particular if it has light tails, then $X_\varphi=X_\eta$ \cite{KKL2016,Keller2019}.) We also generalize the concept of heredity to the non-proximal (and hence non-hereditary) case by proving that $X_\varphi$ is always ``hereditary away from its unique minimal subsystem'' (which is alway Toeplitz). Finally we characterize regularity of this Toeplitz subsystem equivalently by the condition $m_H\left(\partial(\overline{\inn(W)})\right)=0$, where $W$ (``the window'') is a subset of a compact abelian group $H$ canonically associated with the set $\cB$, and $m_H$ denotes Haar measure on $H$. Throughout, results from \cite{KKL2016} are heavily used.
\end{abstract}}

\section{Introduction and results}\label{sec:introduction}
For any given non-empty set $\cB\subseteq\N=\{1,2,\dots\}$ one can define its \emph{set of multiples}
\begin{equation*}
\cM_\cB:=\bigcup_{b\in\cB}b\Z
\end{equation*}
and the set of \emph{$\cB$-free numbers}
\begin{equation*}
\cF_\cB:=\Z\setminus\cM_\cB\ .
\end{equation*}
The investigation of structural properties of $\cM_\cB$ or, equivalently, of $\cF_\cB$ has a long history (see the monograph \cite{hall-book} and the recent paper \cite{BKKL2015} for references), and dynamical systems theory provides some useful tools for this. Namely, denote by $\eta\in\{0,1\}^\Z$ the characteristic function of $\cF_\cB$, i.e. $\eta(n)=1$ if and only if $n\in\cF_\cB$, and consider the
orbit closure $X_\eta$ of $\eta$ in the shift dynamical system
$(\{0,1\}^\Z,\sigma)$, where $\sigma$ stands for the left shift. Then topological dynamics and ergodic theory provide a wealth of concepts to describe various aspects of the structure of $\eta$, see \cite{Sa} which originated
 this point of view by studying the set of square-free numbers, and also
\cite{Peckner2012}, \cite{Ab-Le-Ru}, \cite{BKKL2015}, \cite{KKL2016} for later contributions.

\subsection{Motivation and main results} 
In this note we always assume that $\cB$ is \emph{primitive}, i.e. that there are no $b,b'\in\cB$ with $b\mid b'$. We say that $X_\eta$ is \emph{hereditary}, if $y\in\{0,1\}^\Z$ belongs to $X_\eta$ whenever there is $x\in X_\eta$ with $y\leqslant x$.
For certain hereditary $\cB$-free systems the structure of invariant measures, in particular the uniqueness of the measure of maximal entropy, was studied in \cite{KLW2015} and \cite{BKKL2015}, see also \cite{KPM2019}. Here we apply some of the tools developed in \cite{KKL2016} to generalize the concept of heredity.

Obviously, if $X_\eta$ is hereditary, then
$(\dots,0,0,\dots)\in X_\eta$,
in which case this fixed point of the shift is the unique minimal subset of $(X_\eta,\sigma)$ \cite[Thm.~B]{BKKL2015}.
Theorem~3.7 of \cite{BKKL2015} and Theorem~C of \cite{KKL2016} provide lists of properties which are all equivalent to 
``$(\dots,0,0,\dots)\in X_\eta$'' and which also motivate to call $\cB$  \emph{proximal} in this case.
It is noteworthy that one of these equivalent properties is ``$\cB$ contains an infinite pairwise co-prime subset''.

Recall that the set $\cB\subseteq\N$ is \emph{taut} if \footnote{Indeed, tautness of $\cB$ is defined in \cite{hall-book} by $\ddelta(\cM_{\cB\setminus\{b\}})<\ddelta(\cM_\cB)$ for each $b\in\cB$, where $\ddelta(\cM_\cB):=\lim_{n\to\infty}\frac{1}{\log n}\sum_{k\leqslant n,k\in\cM_\cB}k^{-1}$
denotes the logarithmic density of this set, which is known to exist and to coincide with $\underline{d}(\cM_\cB)$ by the Theorem of Davenport and Erd\"os \cite{DE1936,DE1951}.}
\begin{equation}
\forall b\in\cB:\ \underline{d}\left(\cM_{\cB\setminus\{b\}}\right)<\underline{d}\left(\cM_\cB\right),
\end{equation}
and that it has \emph{light tails}, if
\begin{equation}
\lim_{K\to\infty}\overline{d}\left(\cM_{\{b\in\cB:b>K\}}\right)=0\ ,
\end{equation}
where $\underline{d}(\cA):=\liminf_{N\to\infty}N^{-1}\card(\cA\cap[1,N])$ and
$\overline{d}(\cA):=\limsup_{N\to\infty}N^{-1}\card(\cA\cap[1,N])$. 
If $\cB$ has light tails, then 
$\cB$ is taut, but the converse does not hold \cite[Sec.~4.3]{BKKL2015}.

It is proved in \cite[Theorem~D]{BKKL2015} that, for sets $\cB$ with light tails,
proximality is not only necessary for heredity, but also sufficient, and this result is extended to taut sets in \cite[Thm.~3]{Keller2019}.
Our main result generalizes the concept of heredity to the non-proximal case, and it highlights the role of the tautness assumption in this result.
In order to formulate it, we 
recall some notions from \cite{KKL2016}.
\begin{itemize}
\item $\Delta:\Z\to\prod_{b\in\cB}\Z/b\Z$, $\Delta(n)=(n,n,\dots)$, denotes the canonical diagonal embedding.
\item $H:=\overline{\Delta(\Z)}$ is a compact abelian group, and we denote by $m_H$ its normalised Haar measure.
\item The \emph{window} associated to $\cB$  is defined as
\begin{equation}\label{eq:W}
W:=\{h\in H: h_b\neq0\ (\forall b\in\cB)\}.
\end{equation}
\item For an arbitrary subset $A\subseteq H$ we define the coding function
$\varphi_A:H\to\{0,1\}^\Z$ by $\varphi_A(h)(n)=1$ if  and only if $h+\Delta(n)\in A$. Of particular interest are the coding functions $\varphi:=\varphi_W$,
$\varphi_{W'}$ and $\varphi_{\inn(W)}$, where $W':=\overline{\inn(W)}$.
Observe that $\varphi_{\inn(W)}\leqslant\varphi_{W'}\leqslant\varphi$, and that $\varphi_{\inn(W)}$ is lower semicontinuous, while $\varphi_{W'}$ and $\varphi$ are upper semicontinuous.
\item
Observe that
$\varphi(h)(n)=1$ if  and only if  $h_b+n\neq0$ mod $b$ for all $b\in\cB$.
\item With this notation $\eta=\varphi(\Delta(0))$ and $X_\eta=\overline{\varphi(\Delta(\Z))}$, so that
$X_\eta\subseteq X_\varphi:=\overline{\varphi(H)}$.
\end{itemize}
Finally we need to introduce some further notation:
\begin{equation}
\strip :=\left\{x\in\{0,1\}^\Z: \exists h\in H\text{ s.t. }\varphi_{W'}(h)\leqslant x\leqslant\varphi(h)\right\},
\end{equation}
and for any closed subshift $X\subseteq\{0,1\}^\Z$,
\begin{equation}
\tilde X:=\{x'\in\{0,1\}^\Z: \exists\, x\in X\text{ s.t. }x'\leqslant x\}.
\end{equation}
is the \emph{hereditary closure} of $X$.
Observe that
$X$ is {hereditary}, if and only if $\tilde{X}=X$.

\begin{remark}\label{rem:closed-strip}
$\strip$ is a shift-invariant subset of $\{0,1\}^\Z$. It is not necessarily closed, because in general $\varphi_{W'}$ is only upper semicontinuous. It is closed when $\varphi_{W'}$ is continuous, i.e. when $W'$ is open and closed.
This applies in particular when $\cB$ is proximal, i.e. when $\inn(W)=\emptyset$ 
\cite{KR2015,KKL2016} and hence also $W'=\emptyset$, in which case $\varphi_{W'}=0$. In that case $\strip=\tilde{X}_\varphi$, see e.g. \cite[Lemma~5.4]{KR2015}.
\end{remark}

\begin{theorem}
\label{theo:hereditary-ext}
If  $\cB$ is primitive, then $\varphi(H)\subseteq\strip \subseteq X_\varphi$, in particular
$X_\varphi=\overline{\strip}$.
\end{theorem}
\noindent This theorem is proved in Subsection~\ref{subsec:main-proofs}.

\begin{corollary}\label{coro:prox}
If $\cB$ is primitive and proximal, then $X_\varphi=\tilde{X}_\varphi$.
\end{corollary}
\begin{proof}
If $\cB$ is primitive and proximal, then $\strip=\tilde{X}_\varphi$ is closed, see Remark~\ref{rem:closed-strip}. Hence $X_\varphi=\overline{\strip}=\tilde{X}_\varphi$ by Theorem~\ref{theo:hereditary-ext}. 
\end{proof}

\begin{corollary}\label{coro:hereditary-ext}
If  $\cB$ is taut (and hence primitive), then $X_\eta=X_\varphi=\overline{\strip}$.
\end{corollary}
\begin{proof}
As $X_\eta=X_\varphi$ whenever $\cB$ is taut (see \cite[Thm.~4]{Keller2019},
which is based on \cite[Prop.~2.2]{KKL2016} and \cite[Prop.~5.11]{BKKL2015}), this follows from Theorem~\ref{theo:hereditary-ext}.
\end{proof}

\begin{remark}
If $\cB$ is taut and proximal, then $X_\eta=X_\varphi$ 
by Corollary~\ref{coro:hereditary-ext},
and Corollary~\ref{coro:prox} implies $X_\eta=\tilde{X}_\eta$, which was conjectured in \cite{BKKL2015}, where the authors prove that $X_\eta=\tilde{X}_\eta$ when $\cB$ is proximal and has light tails \cite[Thm.~D]{BKKL2015}.

Hence, if $\cB$ is taut and proximal, then
some of the results from \cite{Kwietniak2013} apply: $(X_\eta,S)=(X_\varphi,S)$ is distributionally chaotic of type 2
\cite[Thm.~23]{Kwietniak2013}, but not of type 1 \cite[Thm.~24]{Kwietniak2013}. Moreover, for each $t\geqslant0$, the set of all ergodic invariant measures on $X_\eta$ with entropy not exceeding $t$ is arcwise connected w.r.t.~the $\bar d$-metric on the set of all invariant measures \cite[Thm.~6]{KKK2018}.
\end{remark}

\subsection{Notation and further results}
In order to formulate some more detailed corollaries to Theorem~\ref{theo:hereditary-ext} and to prove this theorem, we need to
recall some further notions from the theory of sets of multiples \cite{hall-book}
and also from \cite{KKL2016}.
Let $\cB$ be a non-empty subset of $\N$.
\begin{itemize}
\item The measure $\nu_\eta:=m_H\circ\varphi^{-1}$ on $X_\varphi$ is called the \emph{Mirsky measure} of the system. The point $\eta\in X_\eta$ is quasi-generic for $\nu_\eta$, in particular $\supp(\nu_\eta)\subseteq X_\eta$ \cite[Prop.~E and Thm.~F]{BKKL2015}.
\item If $\cB$ is taut, then $X_\eta=\supp(\nu_\eta)$ \cite[Thm.~4]{Keller2019}. (The reverse implication is proved in \cite[Cor.~1.8]{KPM2019}.)
\end{itemize}
Combined with Theorem~\ref{theo:hereditary-ext} and Corollary~\ref{coro:hereditary-ext}, this proves the following corollary.

\begin{corollary}\label{coro:all-X-ext}
\begin{compactenum}[a)]
\item
If $\cB$ is primitive, then 
\begin{equation}\label{eq:all-inclusions}
\supp(\nu_\eta)
\subseteq 
X_\eta
\subseteq 
X_\varphi
=
\overline\strip .
\end{equation}
\item If $\cB$ is taut, then 
the inclusions in \eqref{eq:all-inclusions} are identities.
\end{compactenum}
\end{corollary}

\begin{question}
It is proved in \cite[Cor.~1.8]{KPM2019} that the first inclusion in \eqref{eq:all-inclusions} is strict, when $\cB$ is not taut.
To which extent can the second inclusion in \eqref{eq:all-inclusions} be strict when $\cB$ is not taut? What can be said about invariant probability measures assigning full mass to any of these set-differences? Are there any such measures?
\end{question}

\begin{remark}
A particular class of non-taut examples to which Corollary~\ref{coro:prox} applies are Behrend sets $\cB$, i.e.~sets for which $\underline d(\cM_\cB)=1$. For such sets $\overline d(\cF_\cB)=m_H(W)=0$, in particular $\inn(W)=\emptyset$, i.e.~$\cB$ is proximal, and hence $X_\varphi$ is hereditary by the corollary. 
Moreover, $\supp(\nu_\eta)=\{(\dots,0,0,0,\dots)\}$, which illustrates how far the first inclusion in \eqref{eq:all-inclusions} can be from equality, and I expect that also the second inclusion is strict.
Here is a particular example:
Denote by $\cP$ the set of prime numbers. Then $\cB:=\{pq: p,q\in\cP\}$ is a Behrend set, and indeed $\cF_\cB=\cP\cup(-\cP)\cup\{-1,1\}$ \cite[p.~173]{Meyer1973}.
\end{remark}

We continue to recall some notation and results from \cite{KKL2016}.
By $S,S'\subset\cB$ we always denote \emph{finite} subsets. For $S\subset\cB$ let
\begin{equation}\label{eq:A_S-def}
\cA_S:=\{\gcd(b,\lcm(S)): b\in\cB\},
\end{equation}
and note that $\cF_{\cA_S}\subseteq\cF_\cB$, because
$b\mid m$ for some $b\in\cB$ implies $\gcd(b,\lcm(S))\mid m$ for any $S\subset\cB$.
Let $\cE:=\bigcup_{S\subset\cB}\cF_{\cA_S}\subseteq\cF_\cB$. Then 
\begin{equation}\label{eq:E-complement}
\Z\setminus\cE
=
\bigcap_{S\subset\cB}\cM_{\cA_S},
\end{equation}
and
\begin{equation}\label{eq:E-intW}
\cE=\{n\in\Z: \Delta(n)\in\inn(W)\}
\end{equation}
by \cite[Lemma~3.1c]{KKL2016}. Observe that $\varphi_{{\inn(W)}}(\Delta(k))= \sigma^k(1_\cE)$ for all $k\in\Z$.
In view of \cite[Lemma 3.5]{KKL2016}, which states $\Delta(\Z)\cap\left(W'\setminus\inn(W)\right)=\emptyset$, we have $\varphi_{W'}(\Delta(k))=\varphi_{{\inn(W)}}(\Delta(k))=\sigma^k(1_\cE)$ for all $k\in\Z$,
in particular $\varphi_{W'}$ and $\varphi_{\inn(W)}$ are continuous at all points from $\Delta(\Z)\subset H$.

\begin{corollary}\label{coro:cE}
$1_\cE=\varphi_{W'}(\Delta(0))\in X_\varphi$.
\end{corollary}

\begin{proof}
$1_\cE=\varphi_{W'}(\Delta(0))\in\strip $ by definition, and $\strip\subseteq X_\varphi$ by Theorem~\ref{theo:hereditary-ext}.
\end{proof}

Denote the orbit closure of $1_\cE\in\{0,1\}^\Z$ by $X_\cE$.
Then $X_\cE=\overline{\varphi_{{\inn(W)}}(\Delta(\Z))}=\overline{\varphi_{W'}(\Delta(\Z))}$ is 
minimal in view of \cite[Lemma~3.9]{KKL2016}. As $1_\cE\in X_\varphi$, this proves the following corollary.

\begin{corollary}
$X_\cE$ is the unique minimal subset of $X_\varphi$ (and hence also of $X_\eta$).
\end{corollary}

%
%
%
%
%
\begin{remark}
As $\varphi_{\inn(W)}\leqslant\varphi_{W'}$ and as $\varphi_{\inn(W)}$ is lower semicontinuous,
we have
\begin{equation*}
\strip
\subseteq
X_\varphi
=\overline{\strip}
\subseteq
\overline{[\varphi_{\inn(W)},\varphi]}
=
[\varphi_{\inn(W)},\varphi].
\end{equation*}
Assume now that $m_H(\partial W')=0$. (We will see in Proposition~\ref{prop:toeplitz} that this happens if and only if $1_\cE$ is a regular Toeplitz sequence.)
Then $\varphi_{\inn(W)}=\varphi_{W'}$ $m_H$-a.s., so that
$\strip\cap(\piH)^{-1}\{h\}=X_\varphi\cap(\piH)^{-1}\{h\}$ for $m_H$-a.a. $h\in H$.
Hence all invariant measures on $X_\varphi$ are supported by $\strip$, if $m_H(\partial W')=0$.
\end{remark}

\begin{conjecture}
Suppose that $m_H(\partial W')=0$. 
\begin{compactenum}[a)]
\item
For any invariant probability measure $\nu$ on $(X_\varphi,\sigma)$ there exists an invariant probability measure $\rho$ on $(H\times \{0,1\}^\Z,R_{\Delta(1)}\times\sigma)$, whose projection onto the first coordinate equals $m_H$ and such that, for any measurable $A\subseteq X_\varphi$,
\begin{equation}\label{eq:convolution-1}
\nu(A)=\int_{H \times\{0,1\}^\Z} 1_A(\varphi_{W'}(h)+ x\cdot(\varphi(h)-\varphi_{W'}(h)))\,d\rho(h,x).
\end{equation}
\item
$(X_\varphi,S)$ has a unique measure
of maximal entropy. It can be represented as in \eqref{eq:convolution-1} with $\rho=m_H\otimes B_{(1/2,1/2)}$, where $B_{(1/2,1/2)}$ is the Bernoulli measure.
\end{compactenum}
For the hereditary case this is proved in \cite[Theorems~ I and J]{BKKL2015}.
\end{conjecture}

\subsection{The Toeplitz subsystem}
We finish this introduction with a result that sheds some more light on the
unique minimal set $X_\cE=\overline{\varphi_{W'}(\Delta(\Z))}$. Observe that $\overline{\inn(W')}=W'$. If $W'$ were an ``arithmetic window'', just as the window $W$ defined in \eqref{eq:W}, then Theorem B of \cite{KKL2016} would provide much additional information on $X_\cE$, e.g. that $\varphi_{W'}(\Delta(0))$ is a Toeplitz sequence. However, $W'$ is not ``arithmetic'' in the strict sense. But one can modify the set $\cB$ into a set $\cB^*$ in such a way that $1_{\cF_{\cB^*}}=\varphi_{W'}(\Delta(0))$.
To that end we recall the definition of the set $\Ainf$ from \cite{KKL2016},
\begin{equation}\label{eq:Ainf-def}
\Ainf:=\{n\in\N: \forall_{S\subset\cB}\ \exists_{S': S\subseteq S'\subset\cB}: n\in\cA_{S'}\setminus S'\},
\end{equation}
and one of its important properties \cite[Lemma~3.3]{KKL2016},
\begin{equation}\label{eq:E-identities}
\cE=\cF_{\cB\cup{\Ainf}}=\cF_\cB\cap\cF_{\Ainf}.
\end{equation}
Let
\begin{equation}\label{eq:B_0-def}
\cB_0:=\cB\setminus\cM_{\Ainf},\quad
\cB^*:=\cB_0\cup\prim\Ainf,
\end{equation}
where $\prim\Ainf$ denotes the set of primitive elements of $\Ainf$, i.e. the set of those $a\in\Ainf$ which are not multiples of any other element of $\Ainf$.
In Subsection~\ref{subsec:lemmas-proof} we prove the following lemmas:

\begin{lemma}\label{lemma:Ainf}
$n\in \Ainf$ if and only if there is an infinite pairwise co-prime set $\cC\subseteq\N\setminus\{1\}$ such that
$n\,\cC\subseteq\cB$.
\end{lemma}

Let $\Delta^*:\Z\to\prod_{b^*\in\cB^*}\Z/b^*\Z$ denote again the canonical diagonal embedding and let $H^*:=\overline{\Delta^*(\Z)}$.
Define a group homomorphism $\Gamma:\Delta(\Z)\to \Delta^*(\Z)$ by $\Gamma(\Delta(n))=\Delta^*(n)$. 
\begin{lemma}\label{lemma:Gamma-continuous}
$\Gamma$ is uniformly continuous on $\Delta(\Z)$ and extends uniquely to a
continuous surjective group homomorphism $\Gamma:H\to H^*$.
\end{lemma}

Let $W^*\subseteq H^*$ denote the corresponding window and $\varphi^*:H^*\to\{0,1\}^\Z$ the corresponding coding map. 

\begin{lemma}\label{lemma:B^*}(See also \cite[Thm.~B]{KKL2016}.)
\begin{compactenum}[a)]
\item The set $\cB^*$ is primitive.
\item $\cM_{\cB^*}=\cM_{\cB\cup\Ainf}=\bigcap_{S\subset\cB}\cM_{\cA_S}$. Equivalently, $\cF_{\cB^*}=\cF_{\cB\cup\Ainf}=\cE$.
\item There are no $n\in\N$ and no infinite pairwise
co-prime set $\cC\subseteq\N\setminus\{1\}$ such that $n\,\cC\subseteq\cB^*$.
\item The window $W^*\subseteq H^*$ constructed from $\cB^*$ is topologically regular, i.e. $W^*=\overline{\inn(W^*)}$.
\item $1_{\cF_{\cB^*}}=\varphi_{W'}(\Delta(0))=1_\cE$ is a Toeplitz sequence.
\end{compactenum}
\end{lemma}

\begin{lemma}\label{lemma:Gamma}
\begin{compactenum}[a)]
\item $\Gamma(W')=W^*$.
\item 
$\Gamma(H\setminus W')=H^*\setminus W^*$
\item If $(\varphi_{W'}(h))_i=0$ for some $h\in H$ and $i\in\Z$, then 
\begin{equation*}
\begin{split}
&\exists b\in\cB\ \exists b^*\in\cB^*: b\in b^*\Z,\;(h_b+i)\in b^*\Z,
\text{ and [$b^*\in\prim\Ainf$ or $b^*= b$].}
\end{split}
\end{equation*}
\end{compactenum}
\end{lemma}

In view of \cite[Proposition~1.2]{KKL2016}, the Toeplitz sequence $1_{\cF_{\cB^*}}$ is regular if and only if $m_{H^*}(\partial W^*)=0$, where $m_{H^*}$ denotes the Haar measure on $H^*$. The following proposition characterizes the regularity of  
$1_{\cF_{\cB^*}}$ in terms of $m_H(\partial W')$.
It is proved in Subsection~\ref{subsec:toeplitz-proofs}.

\begin{proposition}\label{prop:toeplitz}
The Toeplitz sequence $1_{\cF_{\cB^*}}=\varphi_{W'}(\Delta(0))=1_\cE$ is regular if and only if $m_H(\partial W')=0$.
\end{proposition}

\section{Proofs}
\subsection{Proofs of Lemmas~\ref{lemma:Ainf} -~\ref{lemma:Gamma}}\label{subsec:lemmas-proof}
All sets and quantities derived from $\cB^*$ instead of $\cB$ also carry a $^*$ superscript.

\begin{proof}[Proof of Lemma~\ref{lemma:Ainf}]
Suppose first that $n\in\Ainf$. A repeated application of the definition~\eqref{eq:Ainf-def} of this set allows to produce a strictly increasing sequence $S_1\subset S_2\subset\dots$ of finite subsets of $\cB$ and a strictly increasing sequence $b_1<b_2<\dots$ of elements of $\cB$ such that
\begin{equation*}
n=\gcd\left(\lcm(S_k),b_k\right) \text{ and $b_k\in S_{k+1}$ for all $k\in\N$.}
\end{equation*}
Let $\cC:=\{b_1/n,b_2/n,\dots\}$. Then $n\,\cC=\{b_1,b_2,\dots\}\subseteq\cB$, and $\cC$ is an infinite pairwise co-prime subset of $\N$, because, for any $j<k$,
\begin{equation*}
\gcd(b_j/n,b_k/n)=n^{-1}\gcd(b_j,b_k)\mid n^{-1}\gcd(\lcm(S_k),b_k)=1\ .
\end{equation*}
Finally, $b_j/n\neq 1$ for all $j$, because otherwise $b_{j+1}=b_j/n\cdot b_{j+1}=b_{j+1}/n\cdot b_j$ is a multiple of $b_j$, which is impossible, because $\cB$ is primitive.

Conversely, suppose that $n\,\cC\subseteq\cB$, where $\cC$
is an infinite pairwise co-prime subset of $\N\setminus\{1\}$. 
Then $n\not\in\cB$, because $\cB$ is primitive.
Let $S\subset\cB$ and choose any $c\in\cC$.
Let $S':=S\cup\{nc\}$. As $S'\subset\cB$
is finite,
one can choose some $c'\in\cC$ which is co-prime to $\lcm(S')$.
Then $nc'\in\cB$, and as $n\mid\lcm(S')$, we conclude that
$n=\gcd\left(\lcm(S'),nc'\right)\in\cA_{S'}$. As $n\not\in\cB$, this shows that $n\in\cA_{S'}\setminus S'$. As $S\subset\cB$ could be chosen arbitrarily and as $S'\supseteq S$, this proves that $n\in\Ainf$.
\end{proof}

\begin{proof}[Proof of Lemma~\ref{lemma:Gamma-continuous}]
As $\Gamma$ is a group homomorphism, it suffices to prove continuity at $\Delta(0)\in\Delta(\Z)$ in order to guarantee uniform continuity:
 To see this, fix $b^*\in \cB^*$. If $b^*\in\cB_0$, then $(\Gamma(\Delta(n)))_{b^*}=(\Delta^*(n))_{b^*}=(\Delta(n))_{b^*}$ for all $n\in\Z$.
If $b^*\in \prim\Ainf$,
then there is $S\subset\cB$ such that $b^*\mid\lcm(S)$.
Let $n\in\Z$. In view of the CRT, the residue $n\mod\lcm(S)$ is uniquely determined by the $S$-tuple $(\Delta(n)_b)_{b\in S}$. Hence, as $b^*\mid\lcm(S)$, also 
$(\Gamma(\Delta(n)))_{b^*}=(\Delta^*(n))_{b^*}=n\mod b^*$ is uniquely determined by this $S$-tuple.

Uniform continuity allows to extend $\Gamma$ uniquely to a continuous group homomorphism on $H=\overline{\Delta(\Z)}$. In particular, $\Gamma(H)\supseteq \overline{\Gamma(\Delta(\Z))}=\overline{\Delta^*(\Z)}=H*$.
\end{proof}

\begin{proof}[Proof of Lemma~\ref{lemma:B^*}]
a)\;Let $b,b'\in\cB^*$, $b'=rb$ for some $r\in\N$. We show that $r=1$:
\begin{compactitem}[-]
\item If $b',b\in\prim\Ainf$, then $r=1$, because $\prim\Ainf$ is primitive.
\item If $b'\in\prim\Ainf$ but $b\in\cB_0\subseteq\cB$, then there is some $\tilde b\in\cB$ such that $rb=b'\mid\tilde b$.
As $\cB$ is primitive, this implies $r=1$.
\item If $b'\not\in\prim\Ainf$, then $b'\in\cB_0=\cB\setminus\cM_{\Ainf}$ so that $b\not\in\prim\Ainf$. But then $b\in\cB_0$, and again $r=1$ because $\cB$ is primitive.
\end{compactitem}
b)\;Recall that $\cB^*=(\cB\setminus\cM_{\Ainf})\cup\prim\Ainf$.
As $\cM_{\prim\Ainf}=\cM_{\Ainf}=\cM_{\cM_{\Ainf}}$, we conclude that
\begin{equation*}
\begin{split}
\cM_{\cB^*}
&=
\cM_{\cB\setminus\cM_{\Ainf}}\cup\cM_{\prim\Ainf}
=
\cM_{\cB\setminus\cM_{\Ainf}}\cup\cM_{\cM_{\Ainf}}
=
\cM_{\cB\cup\cM_{\Ainf}}\\
&=
\cM_{\cB}\cup\cM_{\cM_{\Ainf}}
=
\cM_{\cB}\cup\cM_{{\Ainf}}
=
\cM_{{\cB\cup\Ainf}}.
\end{split}
\end{equation*}
Taking complements this is equivalent to $\cF_{\cB^*}=\cF_{\cB\cup\Ainf}=\cE$, where we used \eqref{eq:E-identities} for the last identity. Finally note that $\Z\setminus\cE=\Z\setminus\bigcup_{S\subset\cB}\cF_{\cA_S}=\bigcap_{S\subset\cB}\cM_{\cA_S}$.\\
c)\;Suppose for a contradiction that there
are $n\in\N$ and an infinite pairwise
co-prime set $\cC\subseteq\N\setminus\{1\}$ such that $n\,\cC\subseteq\cB^*$. If $n\in\Ainf$, then 
$n\,\cC\subseteq\cB^*\cap\cM_{\Ainf}=\prim\Ainf$, so that $\cC=\{1\}$, a contradiction. Hence $n\not\in\Ainf$. 
Then Lemma~\ref{lemma:Ainf} implies that 
$(n\,\cC)\cap\cB_0\subseteq (n\,\cC)\cap\cB$ is finite, so that 
$(n\,\cC)\cap\prim\Ainf$ is infinite. This leads to the contradiction $n\in\Ainf$, as can be seen as follows: Let $\{nc_1<nc_2<\dots\}:=(n\,\cC)\cap\prim\Ainf$. For each $nc_k$ there is an infinite pairwise co-prime set $\cC_k\subseteq\N\setminus\{1\}$ such that $nc_k\cC_k\subseteq\cB$ (Lemma~\ref{lemma:Ainf}).
Inductively we choose elements $c_j'\in\cC_{k_j}$ for $j=1,2,\dots$ and suitable $k_1<k_2<\dots$ in the following way: $k_1=1$ and
$c_1'$ can be any element of $\cC_1$. Suppose that $k_1<\dots<k_j$ and $c_1',\dots,c_j'$ are chosen. 
Then choose $k_{j+1}>k_j$ and $c_{j+1}'\in \cC_{k_{j+1}}$ in such a way that $c_{k_{j+1}}$ and
$c_{j+1}'$ are co-prime to the finite set
$\{c_{k_1}c_{1}',\dots,c_{k_j}c_{j}'\}$. In this way we obtain an infinite pairwise co-prime
set $\cC':=\{c_{k_1}c_{1}',c_{k_2}c_{2}',\dots\}\subseteq\N\setminus\{1\}$. It satisfies $n\,\cC'\subseteq\cB$, because 
$n\cdot c_{k_j}c_j'\in nc_{k_j}\cC_{k_j}\subseteq\cB$. Hence $n\in\Ainf$ by Lemma~\ref{lemma:Ainf}.\\
d) \;As $\cB^*$ is primitive (see assertion a), this assertion as well as the Toeplitz property in e) are equivalent to assertion c), see \cite[Thm.~B]{KKL2016}.\\
e) \; The two identities follow from assertion b) and Corollary~\ref{coro:cE}.
\end{proof}

\begin{proof}[Proof of Lemma~\ref{lemma:Gamma}]
a) 
Let $k\in\Z$. Then 
\begin{equation*}
\begin{split}
\Delta^*(k)\in\inn(W^*)
&\Leftrightarrow
k\in\cE^*\hspace*{1.2cm}\text{by \cite[Lemma~3.1]{KKL2016},}\\
&\Leftrightarrow
k\in\cF_{\cB^*}\hspace*{1cm}\text{by Lemma~\ref{lemma:B^*}d and \cite[Prop.~3.2]{KKL2016},}\\
&\Leftrightarrow
k\in\cE\hspace*{1.35cm}\text{by Lemma~\ref{lemma:B^*}b.}
\end{split}
\end{equation*}
Observing the topological regularity of $W^*$ (see Lemma~\ref{lemma:B^*}d), we conclude
\begin{equation*}
W^*
=
\overline{\inn(W^*)\cap\Delta^*(\Z)}
=
\overline{\Delta^*(\cE)}
=
\overline{\Gamma(\Delta(\cE))}
=
\Gamma(\overline{\Delta(\cE)}),
\end{equation*}
where we used the continuity of $\Gamma$ and the compactness of $\overline{\Delta(\cE)}$ for the last identity. Hence
\begin{equation*}
W^*
=
\Gamma(\overline{\Delta(\cE)})
=
\Gamma(\overline{\inn(W)\cap\Delta(\Z)})
=
\Gamma(\overline{\inn(W)})
=
\Gamma(W').
\end{equation*}
b)
Let $h\in H\setminus W'$ and suppose for a contradiction that $\Gamma(h)\in W^*$. We will use the following facts:
\begin{compactenum}[(1)]
\item  Because of part a) of this lemma, there is $h'\in W'$ such that $\Gamma(h')=\Gamma(h)$. As $W'$ is topologically regular by construction and as $\Delta(\Z)$ is dense in $H$, there are $n'_1<n'_2<\dots$ such that $h'=\lim_{j\to\infty}\Delta(n'_j)$ and hence $\Gamma(h')=\lim_{j\to\infty}\Gamma(\Delta(n'_j))=\lim_{j\to\infty}\Delta^*(n'_j)$, where all $\Delta(n'_j)$ belong to $\inn(W')=\inn(W)$. Then \cite[Lem.~3.1c]{KKL2016} implies $n'_j\in\cE$, so that all $n'_j\in\cF_{\cB^*}$ by Lemma~\ref{lemma:B^*}b.
\item Fix some sequence $n_1<n_2<\dots$ such that $h=\lim_{j\to\infty}\Delta(n_j)$.
As $W'$ is closed, there exists some neighbourhood of $h$ which is disjoint from $W'$. Hence there are a finite subset $S$ of $\cB$ and some integer $j_1$ such that for all $j\geqslant j_1$
\begin{compactitem}[-]
\item $(\Delta(n_j))_b=h_b$ for all $b\in S$, and
\item if for some $k\in\Z$ holds $(\Delta(n_j))_b=(\Delta(k))_b$ for all $b\in S$, then $\Delta(k)\not\in W'$ and hence $k\in\cM_{\cB^*}$.
\end{compactitem}
(For the latter conclusion recall that
$k\in\cM_{\cB^*}=\Z\setminus\cE$ if and only if $\Delta(k)\not\in\inn(W)=\inn(W')$ by \cite[Lem.~3.1c]{KKL2016}.)
\item As $\lim_{j\to\infty}\Delta^*(n_j)=\lim_{j\to\infty}\Gamma(\Delta(n_j))=\Gamma(h)=\Gamma(h')=\lim_{j\to\infty}\Delta^*(n'_j)$, there is $j_2\geqslant j_1$ such that $(\Delta^*(n_j))_{b^*}=(\Delta^*(n'_j))_{b^*}$ for all $b$ from the finite set $\cB^*_S:=\{b^*\in\cB^*:b^*\mid\lcm(S)\}$ and all $j\geqslant j_2$.
\end{compactenum}
We first exploit fact (2): Let $j\geqslant j_1$. For each $k=n_j+t\cdot\lcm(S)$
there are $r_{j,t}\in\Z$ and $b_{j,t}^*\in\cB^*$ such that $k=r_{j,t}\cdot b^*_{j,t}$. Let $s_j:=\gcd(n_j,\lcm(S))$. By Dirichlet's theorem on primes in arithmetic progressions, given $j\geqslant j_1$, there are infinitely many primes $p_1<p_2<\dots$ and integers $t_i$ such that $s_j\,p_i=n_j+t_i\,\lcm(S)=r_{j,t_i}\,b^*_{j,t_i}$. For each $p_i$ at least one of the following two possibilities occurs:
\begin{compactitem}[-]
\item $p_i\mid r_{j,t_i}$. Then $b^*_{j,t_i}$ divides $s_j$ and hence $\lcm(S)$.
\item $p_i\mid b^*_{j,t_i}$. Then $r_{j,t_i}$ divides $s_j$ so that $q_i:=\frac{s_j}{r_{j,t_i}}\mid\lcm(S)$ and $q_i=\frac{b^*_{j,t_i}}{p_i}$.
\end{compactitem}
If, for any given $j\geqslant j_1$, the  second possibility occurs for infinitely many $p_i$, then there is at least one $q$ such that $q=q_i$ for infinitely many $p_i$. Hence $q\,p_i=b^*_{j,t_i}\in\cB^*$ for all these $p_i$, which implies that $W^*$ is not topologically regular \cite[Thm.~B]{KKL2016} and thus contradicts
our Lemma~\ref{lemma:B^*}d. Otherwise, for all $j\geqslant j_1$, the first possibility occurs  for all but at most finitely many $p_i$. In particular, for each $j\geqslant j_1$, there exists $i_j$ such that $n_j=r_{j,t_{i_j}}\,b^*_{j,t_{i_j}}-t_{i_j}\,\lcm(S)$, and $b^*_{j,t_{i_j}}\mid\lcm(S)$. Abbreviate $b^*_j:=b^*_{j,t_{i_j}}$. It follows that $n_j\in b^*_j\Z$ and $b_j^*\in\cB^*_S$, i.e.~$(\Delta^*(n_j))_{b_j^*}=0$ and $b_j^*\in\cB^*_S$. Now fact (3) implies that also $(\Delta^*(n'_j))_{b_j^*}=0$, i.e.~$n'_j\in b_j^*\Z\subseteq\cM_{\cB^*}$, which contradicts fact (1).
\\
c) Let $h\in H$. Then $h\in H\setminus W'$ if and only if $\Gamma(h)\in H^*\setminus W^*$ because of assertions a) and b) of this lemma.
Hence, for any $i\in\Z$,
%
\begin{equation*}
\begin{split}
(\varphi_{W'}(h))_i=0
\Leftrightarrow&\;
h+\Delta(i)\in H\setminus W'\\
\Leftrightarrow&\;
\Gamma(h+\Delta(i))\in  H^*\setminus W^*\\
\Leftrightarrow&\;
\exists b^*\in\cB^*: (\Gamma(h+\Delta(i)))_{b^*}=0\\
\Leftrightarrow&\;
\exists b^*\in\cB^*: (\Gamma(h))_{b^*}+i\in b^*\Z.
\end{split}
\end{equation*}
If $b^*\in\prim\Ainf$, there is some $b\in\cB$ such that $b^*\mid b$. If $b^*\in\cB_0$, choose $b=b^*$. In any case we have $b\in b^*\Z$.
Pick a sequence $(m_j)_j$ in $\Z$ such that $h=\lim_{j\to\infty}\Delta(m_j)$.
Then, for all sufficiently large $j\in\N$,
\begin{equation*}
(\Gamma(h))_{b^*}
=
(\Gamma(\Delta(m_j)))_{b^*}
=
(\Delta^*(m_j))_{b^*}
=
\left(m_j\mod{b^*}\right)
=
\left(m_{j}\mod{b}\right)\mod b^*
=
h_b\mod b^*,
\end{equation*}
so that
\begin{equation*}
(\varphi_{W'}(h))_i=0
\Leftrightarrow\;
(h_b+i)\in b^*\Z.
\end{equation*}
\end{proof}

\subsection{Proofs of Theorem~\ref{theo:hereditary-ext} and Corollary~\ref{coro:all-X-ext}}
\label{subsec:main-proofs}
The proof of Theorem~\ref{theo:hereditary-ext} is inspired by the proof of Theorem~D in \cite{BKKL2015} and combines ideas from that proof with observations from \cite{KKL2016}. The two main differences compared to the proofs in \cite{BKKL2015} are:
\begin{itemize}[-]
\item The proximality assumption enters the proof from \cite{BKKL2015} immediately after eqn.~(58) of that paper: for each finite $S\subset\cB$ there exists $b\in\cB$ such that $\gcd(\lcm(S),b)=1$, i.e. $1\in\cA_S$ and hence $\Z=\cM_{\cA_S}$. Without that assumption we can only use our general knowledge of the set $\cA_S$,
namely that $\Z\setminus\cE\subseteq\cM_{\cA_S}$.
\item The light tails assumption is used in the proof of \cite[Lemma~5.21]{BKKL2015}, from which Proposition~5.11 of that reference is deduced. 
This proposition 
basically asserts that for certain cylinder sets in $H$ there exist $k'\in\Z$ such that $\sigma^{k'}\eta=\varphi(\Delta(k'))$ belongs to that cylinder set. This is proved by showing that the set of integers $k'$ with this property has positive density.
Here we prove the existence by a topological argument using compactness of $H$ and the Heine-Borel theorem. The price to pay is that the existence can no longer 
be guaranteed in the set $\Delta(\Z)$ but only in its closure $H$.
\end{itemize}
The following notation will be used repeatedly:
\begin{itemize}
\item By $S,S'\subset\cB$ we always mean \emph{finite} subsets.
\item The topology on $H$ is generated by the (open and closed) cylinder sets
\begin{equation*}
U_S(h):=\{h'\in H: \forall b\in S: h_b=h'_b\},\text{ defined for finite }S\subset\cB\text{ and }h\in H\ .
\end{equation*}
\item For $n\geqslant1$, let $\cB^{(n)} := \{b\in\cB : p\leqslant n\text{ for any }p\in\Spec(b)\}$,
where $\Spec(b)$ stands for the set of all prime divisors of $b$. 
\item For $n\in\N$ and $b\in\cB$ let $H_b^N:=\bigcap_{i=-N}^N\{h\in H: h_b+i\neq 0\mod b\}$. 
All sets $H_b^N$ are open and closed by definition. If  $b\not\in\cB^{(2N+1)}$, then $b$ has at least one prime divisor greater than $2N+1$ so that 
$\Z\not \subseteq\bigcup_{i=-N}^N(b\Z+i)$, in particular
$H_b^N\neq\emptyset$.
For $\cC\subseteq\cB$ denote $H_\cC^N:=\bigcap_{b\in\cC}H_b^N$.
\end{itemize}
We start with a modification of some lemmas from \cite{BKKL2015}.

\begin{lemma}\label{lemma:new-1}
Let $N\in\N$, $n\geqslant 2N+1$, $\cB^{(n)}\subseteq \cA\subset\cB$, $\cA$ finite, $p>n$ a prime number not dividing $\lcm(\cA)$,  and $S\subset\cB\setminus\cA$ be finite. Denote $S_0:=\{b\in S: p\not\mid b\}$. If 
$U_\cA(\Delta(k))\cap H_{S_0}^N\neq\emptyset$ for some $k\in\Z$, then also 
$U_\cA(\Delta(k))\cap H_S^N\neq\emptyset$.
\end{lemma}

\begin{proof}
If $S_0=S$, there is nothing to prove. So suppose that $S\setminus S_0\neq \emptyset$.
By definition,
\begin{equation*}
U_\cA(\Delta(k))\cap H_{S_0}^N
\subseteq
(U_\cA(\Delta(k))\cap H_S^N)\cup\bigcup_{b\in S\setminus S_0}\bigcup_{i=-N}^N\{h\in U_\cA(\Delta(k)): h_b+i\in p\Z\}.
\end{equation*}
Fix any $b'\in S\setminus S_0$. As all $b\in S\setminus S_0$ are multiples of $p$, we have 
$h_b-h_{b'}\in p\Z$ for such $b$, so that
$h_b+i\in p\Z$ if and only if $h_{b'}+i\in p\Z$ for all $h\in H$ and $i=-N,\dots,N$. Hence
\begin{equation*}
U_\cA(\Delta(k))\cap H_{S_0}^N
\subseteq
(U_\cA(\Delta(k))\cap H_S^N)\cup\bigcup_{i=-N}^N\{h\in U_\cA(\Delta(k)): h_{b'}+i\in p\Z\}.
\end{equation*}

Suppose for a contradiction that $U_\cA(\Delta(k))\cap H_S^N=\emptyset$. Then
\begin{equation}\label{eq:assumption-1}
U_\cA(\Delta(k))\cap H_{S_0}^N
\subseteq
\bigcup_{i=-N}^N\{h\in U_\cA(\Delta(k)): h_{b'}+i\in p\Z\}.
\end{equation}
As $p>n\geqslant 2N+1$, there is $j\in\Z$ such that $j+i\not\in p\Z$ for all $i=-N,\dots,N$.
Let $h\in U_\cA(\Delta(k))\cap H_{S_0}^N$. 
As $\gcd(p,b)=1$ for all $b\in\cA\cup S_0$, the CRT guarantees the existence of an integer $m$ such that 
\begin{compactitem}[-]
\item $m=h_b=k\mod b$ for all $b\in \cA$,
\item $m=h_b\mod b$ for all $b\in S_0$, and
\item $m=j\mod p$.
\end{compactitem}
Hence $\Delta(m)\in U_\cA(\Delta(k))\cap H_{S_0}^N$, but since $p|b'$, we have
$\Delta(m)|_{b'}+i\mod p=j+i\mod p\neq0$ for all $i=-N,\dots,N$ in contradiction to \eqref{eq:assumption-1}.

\end{proof}

\begin{lemma}\label{lemma:new-2}
Let $N\in\N$, $n\geqslant 2N+1$, and $\cB^{(n)}\subseteq \cA\subset\cB$, $\cA$ finite.
Denote 
\begin{equation}\label{eq:cC-def}
\cC=\cC(\cA,n):=\{b\in\cB\setminus\cA: \text{$b$ has at least one prime divisor $p>n$  co-prime to $\lcm(\cA)$}\}.
\end{equation}
Then, for each $h\in H$, there is $h'\in U_\cA(h)$ such that
$h'_b+i\neq 0\mod b$ for all $b\in\cC$ and $i\in\{-N,\dots,N\}$.
\end{lemma}

\begin{proof}
We must show that $U_\cA(h)\cap H_{\cC}^N\neq\emptyset$. As all $H_b^N$ are compact, 
it suffices to show that $U_\cA(h)\cap H_S^N\neq\emptyset$ 
for each finite set $S\subset\cC$. 
We proceed by induction on the cardinality of the set $S$.
\begin{itemize}[-]
\item  If $S=\emptyset$, there is nothing to prove, because $U_\cA(h)\neq\emptyset$.
\item  For the inductive step, let 
$\emptyset\neq S\subseteq\cC$, $S$ finite, and note that there exists at least one prime number $p>n$ which is co-prime to $\lcm(\cA)$ and divides at least one element $b\in S$. To this $p$ and its associated set $S_0\subseteq S\setminus\{b\}$ we apply Lemma~\ref{lemma:new-1}.
\end{itemize}
\end{proof}

\begin{lemma}(See also \cite[Prop.~5.10]{BKKL2015}.)\label{lemma:ext-1}\\
Let $h\in H$, $N\in\N$ and $I_0\subseteq\{-N\dots,N\}$ be such that 
$(\varphi_{W'}(h))_{i}=0<1=(\varphi(h))_i$ for all $i\in I_0$. Define $\omega\in\{0,1\}^{\{-N,\dots,N\}}$ by
\begin{equation*}
\omega_i
:=
\begin{cases}
(\varphi(h))_i&\text{ if }i\not\in I_0\\
0&\text{ if }i\in I_0.
\end{cases}
\end{equation*}
Then, for each $S\subset\cB$, there is $h'\in U_S(h)$ such that $(\varphi(h'))_i=\omega_i$ for all $i=-N,\dots,N$.
\end{lemma}

\begin{proof}
\textbf{Step 1:}\quad This step applies to any $h\in H$ and $n\geqslant2N+1$.
For any set $\tilde\cB$ with $\cB^{(n)}\subseteq\tilde\cB\subseteq\cB$ denote
\begin{equation*}
\begin{split}
I_N^{\tilde\cB}(h)
:=&
\{i\in\{-N,\dots,N\}: \exists b\in{\tilde\cB}\text{ s.t. }h_b+i\in b\Z\}.
\end{split}
\end{equation*}
Observe that for ${\tilde\cB}=\cB$, 
\begin{equation*}
\begin{split}
I_N^\cB(h)
=
\{i\in\{-N,\dots,N\}:h+\Delta(i)\not\in W\}
=
\{i\in\{-N,\dots,N\}: (\varphi(h))_i=0\}.
\end{split}
\end{equation*}
In particular, $I_0\cap I_N^{\cB}(h)=\emptyset$.

For each $i\in I_N^{\tilde\cB}(h)$ we fix some $b_i\in{\tilde\cB}$ with  $(h_{b_i}+i)\in b_i\Z$.
Let
\begin{equation*}
{\tilde\cB}_1:=\cB^{(n)}\cup\{b_i: i\in I_N^{\tilde\cB}(h)\}\subseteq{\tilde\cB}.
\end{equation*}
By \cite[Lemma~5.14]{BKKL2015}, the set ${\tilde\cB}_1$ is finite. 

Let $\tilde\beta_1:=\lcm({\tilde\cB}_1)$. Then, for any $i\in I_N^{\tilde\cB}(h)$ and $\ell\in\Z$,
$(h_{b_i}+i+\ell\tilde\beta_1)\in b_i\Z$, i.e.
$I_N^{\tilde\cB}(h)\subseteq I_N^{{\tilde\cB}_1}(h+\Delta(\ell\tilde\beta_1))$.
Assume for a contradiction that there exists $j\in I_N^{{\tilde\cB}_1}(h+\Delta(\ell\tilde\beta_1))\setminus I_N^{\tilde\cB}(h)$. Then there is $b\in{\tilde\cB}_1$ such that $(h_{b}+\ell\tilde\beta_1+j)\in b\Z$, but $(h_{b}+j)\not\in b\Z$, a contradiction to $b\mid\tilde\beta_1$. Therefore,
\begin{equation}\label{eq:I_N-identity}
I_N^{\tilde\cB}(h)= I_N^{{\tilde\cB}_1}(h+\Delta(\ell\tilde\beta_1))\quad\text{for all }\ell\in\Z.
\end{equation}
\textbf{Step 2:}\quad
From now on we choose $n\geqslant2N+1$ large enough that $S\subseteq\cB^{(n)}$, and we consider the particular element $h\in H$ from the assumption of the lemma.

For each $i\in I_0$ we have $(\varphi_{W'}(h))_i=0$. 
{
Hence
one can choose  $\tilde b_i\in\cB$ and $\tilde b_i^*\in\cB^*$ 
as in Lemma~\ref{lemma:Gamma}c
such that $\tilde b_i\in\tilde b_i^*\Z$ and
$h_{\tilde b_i}+i\in\tilde b_i^*\Z$. But $h_{\tilde b_i}+i\not\in\tilde b_i\Z$, because $(\varphi(h))_i=1$.
In particular, $\tilde b_i^*\neq \tilde b_i$ so that $\tilde b_i^*\in\prim\Ainf$ by Lemma~\ref{lemma:Gamma}c. 
Hence, in view of Lemma~\ref{lemma:Ainf}, there is an infinite pairwise co-prime set $\cD_i\subseteq\N\setminus\{1\}$ such that $\tilde b_i^*\cD_i\subseteq\cB$. 
}

We are going to apply Step 1 to $\tilde\cB=\cB$. This yields the set $\cB_1=\cB^{(n)}\cup\{b_i:i\in I_N^\cB(h)\}$ and the number $\beta_1:=\lcm(\cB_1)$.
As $I_0$ is a finite set, we can choose $d_i\in\cD_i$ $(i\in I_0)$ pairwise co-prime and co-prime to $\beta_1$ and all $\tilde b_j$, $j\in I_0$, 
{
in particular also co-prime to all $\tilde b_j^*$, $j\in I_0$.
Denote $b_i:=d_i\cdot\tilde b_i^*$ for $i\in I_0$.} 
\footnote{Observe that the choice of the $b_i$, $i \in I_0$, is not in conflict with the $b_i$ chosen in Step 1, because in that step the indices $i$ are in the set $I_N^{{\cB}}(h)$, which is disjoint from $I_0$.} Then $b_i\in\cB$; and
$b_i>n$, because otherwise $b_i\in\cB^{(n)}\subseteq\cB_1$, 
so that $d_i\mid b_i\mid\beta_1$ in contradiction to the choice of the $d_i$.
As $\tilde b_i^*\mid\gcd(b_i,\tilde b_i)$, we see that $\tilde b_i^*\mid (h_{\tilde b_i}-h_{b_i})$.
It follows that $h_{b_i}+i\in\tilde b_i^*\Z$.

We claim that there is $x\in\Z$ such that
\begin{align*}
x&=-(h_{b_i}+i)\mod b_i\quad\text{for all }i\in I_0,\\
x&=0\hspace{1.27cm}\mod\beta_1.
\end{align*}
Indeed, this follows from the CRT, because
{
\begin{compactitem}[-]
\item $\gcd(b_i,b_j)=\gcd(d_i\,\tilde{b}_i^*,d_j\,\tilde{b}_j^*)=\gcd(\tilde{b}_i^*,\tilde{b}_j^*)
\mid (h_{b_j}+j)-(h_{b_i}+i)$
for all $i,j\in I_0$, and
\item $\gcd(b_i,\beta_1)=\gcd(d_i\,\tilde{b}_i^*,\beta_1)=\gcd(\tilde b_i^*,\beta_1)\mid \tilde b_i^*\mid b_i\mid (h_{b_i}+i)$.
\end{compactitem}
}
Hence $x=\ell\beta_1$ for some $\ell\in\Z$, and 
\begin{equation}\label{eq:unique-i-1}
(h+\Delta(x)+\Delta(i))_{b_i}=0\quad\text{for all $i\in I_0$.}
\end{equation}
As $b_i>n\geqslant 2N+1$,
\begin{equation}\label{eq:unique-i-2}
(h+\Delta(x)+\Delta(k))_{b_i}\neq0\quad\text{for all }i\in I_0\text{ and } k\in\{-N,\dots,N\}\setminus\{i\}.
\end{equation}

Let $\cB^\diamondsuit:=\cB_1\cup\{b_i:i\in I_0\}$ and $h^\diamondsuit:=h+\Delta(x)$. We claim that
\begin{equation}\label{eq:I_N-identity-2}
I_N^\cB(h)\cup I_0=I_N^{\cB_1}(h+\Delta(x))\cup I_0=I_N^{\cB^\diamondsuit}(h^\diamondsuit).
\end{equation}
Indeed, the first identity follows from \eqref{eq:I_N-identity}, and the ``$\subseteq$''-direction of the second one follows from \eqref{eq:unique-i-1}.
So let $k\in I_N^{\cB^\diamondsuit}(h^\diamondsuit)$, and fix  $b\in\cB^\diamondsuit$ such that $h^\diamondsuit_b+k\in b\Z$. If $b\in\cB_1$, then $k\in I_N^{\cB_1}(h^\diamondsuit)$, and we are done. Otherwise there is $i\in I_0$ such that $b=b_i$, i.e. $(h+\Delta(x)+\Delta(k))_{b_i}= h^\diamondsuit_{b}+k\mod b\in b\Z$.
In view of \eqref{eq:unique-i-2} this implies $k=i\in I_0$.

\textbf{Step 3:}\quad
In this  step we  replace the set $\cB^\diamondsuit$ in \eqref{eq:I_N-identity-2} by $\cB$. The price to pay will be to replace also $h^\diamondsuit$ by another element from $U_S(h)$.

As $\cB^{(n)}\subseteq\cB_1\subseteq\cB^\diamondsuit\subseteq\cB$, we can apply Step 1 to $\tilde\cB=\cB^\diamondsuit$ and $h=h^\diamondsuit$.
This yields a set $\cB^{(n)}\subseteq\cB^\diamondsuit_1\subseteq\cB^\diamondsuit$ such that
\begin{equation}\label{eq:I_N-identity-3}
I_N^{\cB^\diamondsuit}(h^\diamondsuit)= I_N^{\cB^\diamondsuit_1}(h^\diamondsuit).
\end{equation}
Denote $\beta_1^\diamondsuit:=\lcm(\cB^\diamondsuit_1)$, and
define
\begin{equation}\label{eq:B3-def}
\cB_2:=\{b\in\cB\setminus\cB_1^\diamondsuit: \text{$b$ has no prime divisor $p>n$ co-prime to $\beta_1^\diamondsuit$}\}.
\end{equation}
As $\cB^{(n)}\subseteq\cB_1^\diamondsuit$, the set $\cB$ is the disjoint union of $\cB_1^\diamondsuit$, $\cB_2$ and $\cC(\cB_1^\diamondsuit,n)$, where $\cC(\cB_1^\diamondsuit,n)$ is defined in \eqref{eq:cC-def}. Observe that
a prime $p>n$ is co-prime to $\beta_1^\diamondsuit=\lcm(\cB_1^\diamondsuit)$, if and only if it is co-prime to $\lcm(\cB_1^\diamondsuit\cup\cB_2)$.
Hence
\begin{equation}
\begin{split}
\cC(\cB_1^\diamondsuit\cup\cB_2,n)
&=
\{b\in\cB\setminus(\cB_1^\diamondsuit\cup\cB_2): \text{$b$ has at least one prime divisor $p>n$ co-prime to $\beta_1^\diamondsuit$}\}\\
&=
\{b\in\cB\setminus\cB_1^\diamondsuit: \text{$b$ has at least one prime divisor $p>n$ co-prime to $\beta_1^\diamondsuit$}\}\\
&=
\cC(\cB_1^\diamondsuit,n)
=
\cB\setminus(\cB_1^\diamondsuit\cup\cB_2).
\end{split}
\end{equation}
As in the proof of Proposition~5.11 in \cite{BKKL2015} one
shows that $\cB_2$, and hence also $\cB_1^\diamondsuit\cup\cB_2$, is finite
\footnote{If $p$ is a prime divisor of any
$b'\in\cB_2$ then either $p\leqslant n$, or $p > n$ and $p\mid\beta_1^\diamondsuit$. Hence 
$|\Spec(\cB_2)| < \infty$, and since $\cB_2$ is primitive (it is a subset of the primitive set $\cB$), it must be finite \cite[Lemma~ 5.14]{BKKL2015}.}
 and that
\begin{equation}\label{eq:like68}
\begin{split}
I_N^{\cB_1^\diamondsuit\cup\cB_2}(h^\diamondsuit)
&=
I_N^{\cB_1^\diamondsuit}(h^\diamondsuit).
\end{split}
\end{equation}
Indeed, let $i\in I_N^{\cB_1^\diamondsuit\cup\cB_2}(h^\diamondsuit)$. 
Then there is $b\in\cB_1^\diamondsuit\cup\cB_2$ such that $h^\diamondsuit_{b}+i\in b\Z$. 
If $b\in\cB_1^\diamondsuit$, then $i\in I_N^{\cB_1^\diamondsuit}(h^\diamondsuit)$. Otherwise,
$b\in\cB_2\subseteq\cB\setminus\cB_1^\diamondsuit\subseteq\cB\setminus\cB^{(n)}$ has a prime divisor $p>n$. 
By the definition of $\cB_2$ in \eqref{eq:B3-def}, $p\mid \beta_1^\diamondsuit=\lcm(\cB_1^\diamondsuit)$. As $\cB_1^\diamondsuit=\cB^{(n)}\cup\{b_j: j\in I_N^{\cB^\diamondsuit}(h^\diamondsuit)\}$, it follows that
$p\mid b_j$ for some $j\in I_N^{\cB^\diamondsuit}(h^\diamondsuit)$.
Then $p\mid(h^\diamondsuit_{b_j}+j)$, and since $p\mid\gcd(b,b_j)\mid(h^\diamondsuit_{b}-h^\diamondsuit_{b_j})$, it follows that $p\mid(h^\diamondsuit_{b}+j)$. But we also have $p\mid b\mid (h^\diamondsuit_{b}+i)$,  so that $p\mid(i-j)$, which contradicts $p>n\geqslant2N+1$.

Thus having established \eqref{eq:like68},
we apply Lemma~\ref{lemma:new-2} to the set 
$\cB_1^\diamondsuit\cup\cB_2$
So there is $h'\in U_{\cB_1^\diamondsuit\cup\cB_2}(h^\diamondsuit)$ with $h'_b+ i\neq 0\mod b$ for all $b\in\cC(\cB_1^\diamondsuit\cup\cB_2,n)=\cB\setminus(\cB_1^\diamondsuit\cup\cB_2)$ and $i\in\{-N,\dots,N\}$.
Combined with \eqref{eq:I_N-identity-2}, \eqref{eq:I_N-identity-3} and \eqref{eq:like68} this shows
\begin{equation}\label{eq:I_N-final}
I_N^\cB(h)\cup I_0
=
I_N^{\cB^\diamondsuit}(h^\diamondsuit)
=
I_N^{\cB_1^\diamondsuit}(h^\diamondsuit)
=
I_N^{\cB_1^\diamondsuit\cup\cB_2}(h^\diamondsuit)
=
I_N^{\cB_1^\diamondsuit\cup\cB_2}(h')
=
I_N^{\cB}(h').
\end{equation}
Observe that $h'\in U_S(h^\diamondsuit)=U_S(h)$, because $S\subseteq\cB^{(n)}\subseteq\cB_1\cap\cB_1^\diamondsuit$, so that, for all $b\in S$, $h'_b=h^\diamondsuit_b=h_b+\ell\beta_1\mod b=h_b$.

To conclude the proof, notice  that, in view of \eqref{eq:I_N-final}, for any $i\in\{-N,\dots,N\}$
\begin{equation*}
(\varphi(h'))_i=0
\;\Leftrightarrow\;
i\in I_N^\cB(h')
\;\Leftrightarrow\;
i\in I_N^\cB(h)\cup I_0
\;\Leftrightarrow\;
(\varphi(h))_i=0\text{ or }i\in I_0
\;\Leftrightarrow\;
\omega_i=0.
\end{equation*}
\end{proof}

\begin{proposition}\label{prop:hereditary-ext}
If $\varphi_{W'}(h)\leqslant x\leqslant\varphi(h)$ for some $x\in\{0,1\}^\Z$ and $h\in H$, then there are $h_1,h_2,\dots\in H$ such that $h=\lim_{N\to\infty}h_N$ and
$x=\lim_{N\to\infty}\varphi(h_N)$.
\end{proposition}

\begin{proof}
Fix a filtration $S_1\subset S_2\subset\dots$ of $\cB$ by finite sets $S_k$.
For $N\in\N$, let $I_N:=\{i\in\{-N,\dots,N\}: x_i=0, (\varphi(h))_i=1\}$. 
By Lemma~\ref{lemma:ext-1}, there are $h_N\in U_{S_N}(h)$ $(N\in\N)$ such that $x_i=(\varphi(h_N))_i$ for all $i=-N,\dots,N$. Then $\lim_{N\to\infty}h_N=h$ and 
$\lim_{N\to\infty}\varphi(h_N)=x$.
\end{proof}

\begin{proof}[Proof of Theorem~\ref{theo:hereditary-ext}]
$\varphi(H)\subseteq\strip $ is obvious. So
let $x\in\strip $. 
There is $h\in H$ with $\varphi_{W'}(h)\leqslant x\leqslant\varphi(h)$.
Hence $x\in X_\varphi$ because of Proposition~\ref{prop:hereditary-ext}.
\end{proof}


\subsection{Proof of Proposition~\ref{prop:toeplitz}}\label{subsec:toeplitz-proofs}
All sets and quantities derived from $\cB^*$ instead of $\cB$ also carry a $^*$ superscript.
For $S\subset\cB$ define
\begin{equation}
\cM_{\cB^*|S}:=\cM_{\{b\in\cB^*: b\,\mid\,\lcm(S)\}},
\end{equation}
and observe that $\cM_{\cB^*}=\bigcup_{S\subset\cB}\cM_{\cB^*|S}$.

\begin{lemma}\label{lemma:B*S}
$\cM_{\cB^*|S}=\bigcap_{k\in\Z}\left(\cM_{\cB^*}+k\cdot\lcm(S)\right)$ and
$\Z\setminus\cM_{\cB^*|S}=\cE+\lcm(S)\cdot\Z$.
\end{lemma}

\begin{proof}
Let $n\in\cM_{\cB^*|S}$, so $n=\ell b$ for some $\ell\in\Z$ and  $b\in\cB^*$, $b\mid\lcm(S)$. Then $b\mid n-k\cdot\lcm(S)$ for all $k\in\Z$, i.e. 
$n\in\bigcap_{k\in\Z}\left(\cM_{\cB^*}+k\cdot\lcm(S)\right)$. 

Conversely suppose that $n\in\bigcap_{k\in\Z}\left(\cM_{\cB^*}+k\cdot\lcm(S)\right)$.
Then $n+\lcm(S)\cdot\Z\subseteq\cM_{\cB^*}$, in particular $n\in\cM_{\cB^*}$.
Let $m:=\gcd(\lcm(S),n)$, $\hat n:=n/m$ and $\hat s:=\lcm(S)/m$. Then $\hat n$ and $\hat s$ are co-prime, and by Dirichlet's theorem, $\cP':=\cP\cap(\hat n+\hat s\cdot\Z)$ is an infinite set of prime numbers. Observe that $m\,\cP'\subseteq\cM_{\cB^*}$, hence for all $p\in\cP'$ there are $k_p\in\Z$ and $b_p\in\cB^*$ such that $mp=k_pb_p$.

Suppose for a contradiction that for all $p\in\cP'$ there is $\ell_p\in\Z$ such that $b_p=\ell_pp$. Then $m=k_p\ell_p$, and as $\cP'$ is infinite, there are $\ell\in\Z$ and an infinite subset $\cP_0'$ of $\cP'$ such that $\ell_p=\ell$ for all $p\in\cP_0'$.
Hence $\ell\,\cP_0'=\{b_p:p\in\cP_0'\}\subseteq\cB^*$, which is impossible in view of Lemma~\ref{lemma:B^*}c.
So there is some $p\in\cP'$ that
does not divide $b_p$. Hence $m=k_p/p\cdot b_p\in b_p\Z$, and so $n=\hat{n}m\in b_p\Z$ and $\lcm(S)=\hat{s}m=(\hat{s}\cdot k_p/p)\cdot b_p$, i.e. $b_p\mid\lcm(S)$.

We proved $\cM_{\cB^*|S}=\bigcap_{k\in\Z}\left(\cM_{\cB^*}+k\cdot\lcm(S)\right)$. Taking complements this yields
$\Z\setminus\cM_{\cB^*|S}=\bigcup_{k\in\Z}\left(\cF_{\cB^*}+k\cdot\lcm(S)\right)=\cE+\lcm(S)\cdot\Z$, where we used also Lemma~\ref{lemma:B^*}b for the last identity.
\end{proof}

\begin{lemma}\label{lemma:bdW'}
\begin{compactenum}[a)]
\item Let $n\in\Z$ and $S\subset\cB$. Then
$U_S(\Delta(n))\subseteq\inn(W')$ if and only if $n\in\Z\setminus\cM_{\cA_S}$, and
$U_S(\Delta(n))\cap W'\neq\emptyset$ if and only if $n\in\Z\setminus\cM_{\cB^*|S}$.
\item 
$m_H(\partial W')=\inf_{S\subset\cB}d(\cM_{\cA_S}\setminus\cM_{\cB^*|S})$.
\end{compactenum} 
\end{lemma}

\begin{proof} a)\;Recall that $U_S(\Delta(n))$ is open. As $\inn(W')=\inn(W)$, the first assertion follows from \cite[Lemma~3.1a]{KKL2016}.
For the second one observe first that $U_S(\Delta(n))\cap W'\neq\emptyset$ if and only if $U_S(\Delta(n))\cap \inn(W')\neq\emptyset$. In view of \cite[Lemma~3.1a]{KKL2016} this is equivalent to: there are $S\subset S'\subset\cB$ and $\ell\in n+\lcm(S)\cdot\Z$ such that $\ell\in\cF_{\cA_{S'}}$, which can be equivalently rewritten as
\begin{equation*}
\begin{split}
&
\bigcup_{S\subseteq S'\subset\cB} \cF_{\cA_{S'}}\cap\left(n+\lcm(S)\cdot\Z\right)\neq\emptyset\\
\Leftrightarrow\quad&
\cE\cap\left(n+\lcm(S)\cdot\Z\right)\neq\emptyset\\
\Leftrightarrow\quad&
n\in\left(\cE+\lcm(S)\cdot\Z\right)
=
\Z\setminus\cM_{\cB^*|S}\quad(\text{Lemma~\ref{lemma:B*S}}).
\end{split}
\end{equation*}
b)\;Let $S_1\subset S_2\subset\dots$ be any filtration of $\cB$ by finite sets $S_k$. Then part a) of the lemma implies
\begin{equation*}
m_H(\inn(W'))
=
\lim_{k\to\infty}d(\Z \setminus\cM_{\cA_{S_k}})
=
\sup_{S\subset\cB}d(\Z\setminus\cM_{\cA_{S}})
\end{equation*}
and
\begin{equation*}
m_H(W')
=
\lim_{k\to\infty}d(\Z\setminus\cM_{\cB^*|S_k})
=
\inf_{S\subset\cB}d(\Z\setminus\cM_{\cB^*|S}).
\end{equation*}
Observing that $\cM_{\cB^*|S}\subseteq\cM_{\cB^*}\subseteq\cM_{\cA_{S'}}$ (Lemma~\ref{lemma:B^*}) for all $S,S'\subset\cB$, we conclude that
\begin{equation*}
m_H(\partial W')
=
m_H(W')-m_H(\inn(W'))
=
\lim_{k\to\infty}d(\cM_{\cA_{S_k}}\setminus\cM_{\cB^*|S_k})
=
\sup_{S\subset\cB}d(\cM_{\cA_S}\setminus\cM_{\cB^*|S}).
\end{equation*}
\end{proof}

For $S\subset\cB$ let $S^*:=(S\cap\cB_0)\cup\{a\in\prim\Ainf: S\cap a\cdot(\N\setminus\{1\})\neq\emptyset\}$, where $\cB_0=\cB\setminus\cM_{\Ainf}$ as in \eqref{eq:B_0-def}.
Then $S^*\subset\cB^*$.

\begin{lemma}\label{lemma:inclusions}
For any $S\subset\cB$ there is $S'\subset\cB^*$ with $S^*\subseteq S'$ such that
$\cM_{\cA^*_{S'}}\subseteq\cM_{\cA_S}\subseteq\cM_{\cA^*_{S^*}}$.
\end{lemma}
\begin{proof}
Let $S\subset\cB$.
For $a\in\N$ define $\cB_a:=\{b\in\cB^*:\gcd(\lcm(S),b)=a\}$, and let $\cD_S:=\{a\in\N:\cB_a\neq\emptyset\}$. Then $\cD_S\subseteq\{1,\dots,\lcm(S)\}$ is finite and $\cB^*=\bigcup_{a\in\cD_S}\cB_a$. 
For each $a\in\cD_S$ fix some $b_a\in\cB_a$ and let 
\begin{equation*}
S':=S^*\cup\{b_a:a\in\cD_S\}.
\end{equation*}
Then $S'\subset\cB^*$, and we claim that $\cA^*_{S'}\subseteq\cM_{\cA_S}$, so that
$\cM_{\cA^*_{S'}}\subseteq\cM_{\cA_S}$. Indeed: Let $b\in\cB^*$. There is $a\in\N$ such that $b\in\cB_a$. As $b_a\in\cB_a$ as well, we have $a\mid\gcd(b_a,b)\mid\gcd(\lcm(S'),b)$,
so that $\gcd(\lcm(S'),b)\in a\Z$. It remains to show that $a\in\cA_S$.
\begin{compactitem}[-]
\item If $b\in\cB_a\cap\cB_0$, then $a=\gcd(\lcm(S),b)\in\cA_S$.
\item If $b\in\cB_a\cap\prim\Ainf$, then
there is an infinite pairwise co-prime set $\cC\subseteq\N\setminus\{1\}$ such that $b\,\cC\subseteq\cB$ (Lemma~\ref{lemma:Ainf}).
Choose any $c\in\cC$ co-prime to $\lcm(S)$. Then $bc\in\cB$,
so that $a=\gcd(\lcm(S),b)=\gcd(\lcm(S),bc)\in\cA_S$.
\end{compactitem}

For the the proof of the inclusion $\cM_{\cA_S}\subseteq\cM_{\cA^*_{S^*}}$, let $a=\gcd(\lcm(S),b)\in\cA_S$ with $b\in\cB$. Then $\lcm(S^*)\mid\lcm(S)$ by definition of $S^*$, and, by definition of $\cB^*$, there is $b^*\in\cB^*$ which divides $b$.
Hence $a\in \gcd(\lcm(S^*),b^*)\cdot\Z\subseteq\cM_{\cA^*_{S^*}}$. This proves $\cA_S\subseteq\cM_{\cA^*_{S^*}}$ and hence $\cM_{\cA_S}\subseteq\cM_{\cA^*_{S^*}}$.
\end{proof}

\begin{proof}[Proof of Proposition~\ref{prop:toeplitz}]
We proved in Lemma~\ref{lemma:B^*}e that $1_{\cF_{\cB^*}}=\varphi_{W'}(\Delta(0))=1_\cE$ is a Toeplitz sequence. It is regular if and only if $m_{H^*}(\partial W^*)=0$, see \cite[Prop.~1.2]{KKL2016}, and in view of \cite[Lemma~4.3]{KKL2016} \footnote{There is an obvious misprint in this lemma and its proof: All occurences of $\underline d$ must be replaced by $\bar d$.} this is equivalent to
\begin{equation*}
\inf_{S''\subset\cB^*}\ \bar d(\cM_{\cA^*_{S''}}\setminus\cM_{\cB^*})=0.
\end{equation*}
Here $\bar d$ denotes the upper asymptotic density $\bar d(M):=\limsup_{n\to\infty}n^{-1}\card(M\cap[1,n])$. Similarly, the lower asymptotic density is denoted by $d(M)$, and if the limit exists, which is the case for all sets $M=\cM_A$ with finite $A$, it is denoted by $d(M)$.

On the other hand, $m_H(\partial W')=0$ if and only if 
\begin{equation*}
\inf_{S\subset\cB}\ d(\cM_{\cA_{S}}\setminus\cM_{\cB^*|S})=0
\end{equation*}
by Lemma~\ref{lemma:bdW'}.
The equivalence of both conditions will follow, once we have proved the (more general) identity
\begin{equation}\label{eq:identity}
\inf_{S''\subset\cB^*}\ \bar d(\cM_{\cA^*_{S''}}\setminus\cM_{\cB^*})
=
\inf_{S\subset\cB}\ d(\cM_{\cA_{S}}\setminus\cM_{\cB^*|S})\ .
\end{equation}

So let $S\subset\cB$. Choose $S^*\subseteq S'\subset\cB^*$ as in Lemma~\ref{lemma:inclusions}.
Then, for any $S'\subseteq S''\subset\cB^*$,
\begin{equation*}
\cM_{\cA^*_{S''}}\subseteq\cM_{\cA^*_{S'}}
\subseteq
\cM_{\cA_S},
\end{equation*}
so that $\cM_{\cA^*_{S''}}\setminus\cM_{\cB^*}\subseteq\cM_{\cA_S}\setminus\cM_{\cB^*|S}$.
This proves the ``$\leqslant$'' direction in \eqref{eq:identity}.

Conversely, let $S''\subset\cB^*$. Choose any $S\subset\cB$, for which $S''\subseteq S^*$.
Then, by Lemma~\ref{lemma:inclusions},
\begin{equation*}
\cM_{\cA_S}\subseteq
\cM_{\cA^*_{S^*}}
\subseteq
\cM_{\cA^*_{S''}},
\end{equation*}
so that $\cM_{\cA_S}\setminus\cM_{\cB^*|S}\subseteq(\cM_{\cA^*_{S''}}\setminus\cM_{\cB^*})\cup(\cM_{\cB^*}\setminus\cM_{\cB^*|S})$. 
Hence
\begin{equation*}
\begin{split}
d(\cM_{\cA_S}\setminus\cM_{\cB^*|S})
&\leqslant
\bar d(\cM_{\cA^*_{S''}}\setminus\cM_{\cB^*})+\underline{d}(\cM_{\cB^*}\setminus\cM_{\cB^*|S}).
\end{split}
\end{equation*}
This proves the
``$\geqslant$'' direction in \eqref{eq:identity}, because $\underline{d}(\cM_{\cB^*})=\sup_{S\subset\cB}d(\cM_{\cB^*|S})$ by the theorem of Davenport and Erd\"os, see \cite[eqn.~(0.67)]{hall-book}.
\end{proof}

\end{document}